\documentclass{amsart}

\usepackage[foot]{amsaddr}
\usepackage{lmodern}
\usepackage[T1]{fontenc}
\usepackage[utf8]{inputenc}

\usepackage{overpic}
\usepackage{todonotes}

\usepackage{amsmath,amssymb,amsfonts,amsthm}
\usepackage[nobysame]{amsrefs}
\usepackage{mathtools}
\usepackage{paralist}
\usepackage{algorithmicx}
\usepackage{algorithm}
\usepackage{algpseudocode}

\usepackage{hyperref}
\hypersetup{
  pdfauthor={Johanna Lercher, Hans-Peter Schröcker},
  pdftitle={A Multiplication Technique for the Factorization of Bivariate Quaternionic Polynomials},
  pdfkeywords={multiplication technique, bivariate factorization, necessary
    factorization condition, mechanism science},
  hidelinks,
}

\newcommand{\qi}{\mathbf{i}}
\newcommand{\qj}{\mathbf{j}}
\newcommand{\qk}{\mathbf{k}}
\newcommand{\Cj}[1]{{#1}^\ast}
\newcommand{\No}[1]{#1\Cj{#1}}
\newcommand{\R}{\mathbb{R}}

\renewcommand{\H}{\mathbb{H}}
\renewcommand{\DH}{\mathbb{DH}}
\newcommand{\Hstar}{\mathbb{H}_{\ast 1}}
\newcommand{\SO}[1][3]{\operatorname{SO}(#1)}
\algnewcommand\algorithmicinput{\textbf{Require:}}
\algnewcommand\Input{\item[\algorithmicinput]}

\DeclareMathOperator{\mrpf}{mrpf}

\newtheorem{thm}{Theorem}[section]
\newtheorem{lem}[thm]{Lemma}
\newtheorem{prop}[thm]{Proposition}
\newtheorem{cor}[thm]{Corollary}
\theoremstyle{definition}
\newtheorem{defn}[thm]{Definition}

\theoremstyle{remark}
\newtheorem{rmk}[thm]{Remark}
\newtheorem{example}[thm]{Example}

\title[]{A Multiplication Technique for the Factorization of Bivariate Quaternionic Polynomials}

\date{\today}

\author{Johanna Lercher \and Hans-Peter Schröcker}
\address{Department of Basic Sciences in Engineering Sciences, University of Innsbruck, Technikerstr.~13, 6020 Innsbruck, Austria}
\email{johanna.lercher@uibk.ac.at}
\email{hans-peter.schroecker@uibk.ac.at}


\begin{document}

\begin{abstract}
  We consider bivariate polynomials over the skew field of quaternions, where
  the indeterminates commute with all coefficients and with each other. We
  analyze existence of \emph{univariate} factorizations, that is, factorizations
  with univariate linear factors. A necessary condition for existence of
  univariate factorizations is factorization of the norm polynomial into a
  product of univariate polynomials. This condition is, however, not sufficient.
  Our central result states that univariate factorizations exist after
  multiplication with a suitable univariate real polynomial as long as the
  necessary factorization condition is fulfilled. We present an algorithm for
  computing this real polynomial and a corresponding univariate factorization.
  If a univariate factorization of the original polynomial exists, a suitable
  input of the algorithm produces a constant multiplication factor, thus giving
  an \emph{a posteriori} condition for existence of univariate factorizations.
  Some factorizations obtained in this way are of interest in mechanism science.
  We present an example of a curious closed-loop mechanism with eight revolute
  joints.
\end{abstract}

\maketitle

\section{Introduction}
By $(\H,+,\cdot)$ we denote the skew field of real quaternions with the usual addition and non-commutative multiplication. Let $\H[t]$ be the ring of polynomials in the indeterminate $t$ with multiplication defined by the convention that $t$ commutes with all coefficients. A \emph{fundamental theorem of algebra} also holds true for polynomials in $\H[t]$: Each non-constant univariate quaternionic polynomial admits a factorization with linear factors (c.~f. \cite{gentili08,gordon65,niven41}). Remarkably, such a factorization need not be unique. In general, there exist $n!$ different factorizations with linear factors, where $n$ denotes the degree of the polynomial. This ambiguity is caused by non-commutativity of quaternion multiplication.

Factorization of univariate quaternionic polynomials with linear factors is well understood. This article focuses on bivariate quaternionic polynomials with indeterminates $t$, $s$ that commute with all coefficients and with each other. Denote the thus obtained polynomial ring by $\H[t,s]$. In contrast to the univariate case, little is known about criteria that ensure existence of factorizations of a polynomial $Q \in \H[t,s]$ and the same can be said about algorithms to compute factorizations. We briefly explain an important difference between the univariate and the bivariate case: Denote by $\Cj{Q}$ the conjugate polynomial, obtained by conjugating the quaternion coefficients of $Q$. A necessary factorization condition is factorizability of the norm polynomial $\No{Q} \in \R[t,s]$ of $Q$. The reason for this is multiplicativity of the quaternion norm: $Q=Q_1 Q_2$ with $Q_1, Q_2 \in \H[t,s]$ implies $\No{Q}=(\No{Q_1})(\No{Q_2})$ with $\No{Q_1}, \No{Q_2} \in \R[t,s]$. The norm polynomial of a univariate quaternionic polynomial of degree greater than one is a product of at least two irreducible real polynomials that play an important role in the computation of factorizations. This is in contrast to the bivariate case where factorizability of the norm polynomial is exceptional and not a sufficient condition for existence of factorizations. One instance for this (that we will frequently encounter in this text) is given in~\cite{beauregard93}.

There exist few factorization results for bivariate quaternionic polynomials. The most promising article in this context is \cite{Skopenkov19} by Skopenkov and Krasauskas. They introduce a technique for the factorization of bivariate quaternionic polynomials of bi-degree $(n,1)$, where $n \in \mathbb{N}_0$ is an arbitrary non-negative integer. In \cite{lercher20}, we build on their article and characterize all possible factorizations of quaternionic polynomials of bi-degree $(n,1)$ with univariate linear factors. We are not aware of further publications concerning this topic.

In this article, we broaden the ideas of \cite{Skopenkov19} and state some new factorization results for bivariate quaternionic polynomials. Our contribution deals with factorizations of polynomials $Q \in \H[t,s]$ with univariate linear factors, that is,
\begin{equation*}
  Q=a(u_1-h_1)\cdots(u_k-h_k)
\end{equation*}
with $u_i \in \{t,s\}$ and $a, h_i \in \H$ for $i=1,\ldots,k$. We call these
factorizations \emph{univariate} since each factor of $Q$ is a univariate
quaternionic polynomial in either $t$ or $s$. Univariate factorizations may only
occur if $Q$ satisfies $\No{Q}=PR$ with $P \in \R[t]$, $R \in \R[s]$. We call
this rather restrictive condition the \emph{necessary factorization condition}.
By \cite{Skopenkov19}, this condition is also sufficient for quaternionic
polynomials of bi-degree $(n,1)$. Even this special case is interesting and we
consider the results in \cite{Skopenkov19} as an important contribution in the
development of a factorization theory for bivariate quaternionic polynomials.
The degree restrictions in \cite{Skopenkov19} are quite strong but,
unfortunately, necessary. The polynomial
\begin{equation}
\label{NFCfails}
	B=(t^2-\qi)s^2+(2\qj t)s+(\qi t^2-1)
\end{equation}
was given by Beauregard in \cite{beauregard93}. It is an example that satisfies the
necessary factorization condition $\No{B}=(t^4+1)(s^4+1)$ but is irreducible in
$\H[t,s]$ (c.~f. \cite[p.68--69]{beauregard93} for a proof over rational
quaternions and \cite[Proof of Example 1.5]{Skopenkov19} for a proof over real
quaternions).

We have, however,
\begin{multline}
\label{exmult}
	(t^2+1)B=\left(t+\frac{-\qk-\qj}{\sqrt{2}}\right)\left(s+\frac{1-\qi}{\sqrt{2}}\right)\left(t+\frac{1+\qk}{\sqrt{2}}\right)\\ \left(t+\frac{-1+\qk}{\sqrt{2}}\right)\left(s+\frac{-1+\qi}{\sqrt{2}}\right)\left(t+\frac{\qj-\qk}{\sqrt{2}}\right),
\end{multline}
an identity that was called \emph{surprising} by Skopenkov and Krasauskas
\cite[Example~2.11]{Skopenkov19}. In the present paper we show that similar
identities not only hold for the Beauregard polynomial \eqref{NFCfails} but are
generally true: For any polynomial $Q \in \H[t,s]$ satisfying the necessary
factorization condition there exists a real polynomial $K \in \R[t]$ (or $K \in \R[s]$) such that
$KQ$ admits a univariate factorization. The real polynomial $K$ depends on an
ordering of the irreducible real quadratic factors of $R$ (or $P$, respectively), where $\No{Q} =
PR$ with $P \in \R[t]$ and $R \in \R[s]$. For this reason, we can find up to $m!+n!$
univariate factorizations of real multiples of $Q$, where $(m,n)$ denotes the
bi-degree of $Q$. If $Q$ admits a univariate factorization, a suitable ordering
of irreducible real quadratic factors indeed yields $K = 1$. We provide a sharp
upper bound on the total number of possible univariate factorizations of $Q$ and show how to find all of them.

Factorizability of quaternionic polynomials is an interesting research topic in
its own right. However, polynomials over the quaternions and in particular over
the dual quaternions (Section~\ref{sec:remarkable-example}) received recent
attention because of their relation to kinematics and mechanism science:
Univariate (dual) quaternionic polynomials can be used to represent
one-parametric rational motions \cite{hegedus13} and factorizations correspond
to decompositions of the respective motion into simpler motions. In particular,
factorizations with linear factors yield decompositions into rotations and
translations. In \cite{gallet15,hegedus13,hegedus15,li15,li18c}, this fact has
been used to construct mechanisms following a given motion or curve. From
different factorizations mechanical ``legs'' with revolute (or prismatic) joints
are constructed that fully constrain a mechanism with prescribed rational
end-effector motion or trajectory.

The findings of this article allow the extension of this concept to
two-parametric motions. A first example will be outlined in the concluding
Section~\ref{sec:remarkable-example}. In this context,
\begin{enumerate}[(i)]
\item univariate factorizations are of particular interest as each
  univariate linear factor can be realized by a revolute/translational joint and
\item multiplication with a real polynomial factor is admissible as it does not
  change the underlying rational motion.
\end{enumerate}

The paper is structured as follows: In Section~\ref{sec:preliminaries}, we
settle our notation and recall some known facts concerning factorization theory
of univariate and bivariate quaternionic polynomials. In
Section~\ref{sec:multiplication-trick}, we formulate and prove the central
result on existence of $K \in \R[t]$ (or $K \in \R[s]$, respectively) such that
$KQ$ admits a univariate factorization. The proof is constructive and can be
cast into an algorithm for finding $K$ and a univariate factorization of $KQ$.
In Section~\ref{sec:a-posteriori-condition} we consider the case that the
original polynomial $Q$ admits a univariate factorization. In
Section~\ref{sec:remarkable-example} we present applications to mechanism
science and an outlook to future research.

\section{Preliminaries}
\label{sec:preliminaries}

Let us consider the Clifford algebra $\mathcal{C}\ell_{(3,0,1)}$.\footnote{We
  use a Clifford algebra that is larger than actually needed for defining
  quaternions because we will later (in Section~\ref{sec:remarkable-example})
  also consider dual quaternions.} The index $(3,0,1)$ indicates existence of
four basis elements $e_1, e_2, e_3, e_4$, where the first $3$ basis elements
square to $1$ ($e_1^2=e_2^2=e_3^2=1$) and the last basis element squares to $0$
($e_4^2=0$). By the defining conditions for Clifford algebras, the basis
elements also satisfy
\begin{equation*}
  e_ie_j+e_je_i = 0 \quad \text{for} \quad i \neq j.
\end{equation*}
We use the notation $e_{12\cdots k} \coloneqq e_1e_2\cdots e_k$, $0 \leq k \leq
4$, for products of consecutive basis elements. Let us consider the subalgebra
$\H$ of quaternions which is generated by the elements $e_0$, $e_{23}$,
$e_{31}$, $e_{12}$ which are, in that order, identified with
$1$, $\qi$, $\qj$ and $\qk$. An element $h \in \H$ can be written as
\begin{equation*}
  h=h_0+h_1\qi + h_2\qj + h_3\qk \quad \text{with} \quad h_0, h_1, h_2, h_3 \in \R.
\end{equation*}
The basis elements satisfy the multiplication rules
\begin{equation*}
  \qi^2=\qj^2=\qk^2=\qi \qj \qk = -1.
\end{equation*}
Multiplication of quaternions is not commutative but $\H$ forms at least a division ring. The conjugate of $h$ is $\Cj{h}=h_0 - h_1\qi -
h_2\qj - h_3\qk$, its norm is the real number
$\No{h}=h_0^2+h_1^2+h_2^2+h_3^2$ and the inverse of $h \in \H \setminus \{0\}$ is given by $h^{-1}=\Cj{h}/\No{h}$.

By $\H[t]$ and $\H[t,s]$ we denote the set of univariate and bivariate
quaternionic polynomials, respectively. Addition and scalar multiplication of
quaternionic polynomials are defined in the common way. For multiplication we
additionally assume that the indeterminates commute with the coefficients and
with each other. The conjugate $\Cj{Q}$ of a quaternionic polynomial $Q$ is
defined by conjugating its coefficients. Its norm polynomial is the real
polynomial given by $\No{Q}$. We use the following notations: Let us fix
$m,n \in \mathbb{N}_0$. By $\H_{mn}$ we denote the set of quaternionic polynomials
in $\H[t,s]$ of degree at most $m$ in $t$ and at most $n$ in $s$. By replacing
$m$ or $n$ with the symbol $\ast$, we denote the set of polynomials with
arbitrary degree in the respective variable. For instance, the set of
polynomials of bi-degree $(n,1)$ with arbitrary $n \in \mathbb{N}_0$ is denoted by
$\Hstar$. This notation is taken from \cite{Skopenkov19}.

In this article, we repeatedly refer to known factorization results for
univariate and bivariate polynomials. In particular, we use some crucial results
from \cite{Skopenkov19}. Therefore, we would like to briefly recapitulate the
most important univariate and bivariate factorization results. The respective algorithms are taken from \cite[univariate case]{hegedus13} and \cite[bivariate
case]{Skopenkov19} and adapted to our setting. To be able to highlight the similarities between the
univariate and bivariate results, our explanations may sometimes slightly differ
from the original papers.

\subsection{Univariate Polynomials}
\label{subsec:univariate-factoriation}

Each non-constant univariate quaternionic polynomial $Q \in \H[t]$ admits a
factorization with linear factors which can be found as follows: At first, we
assume that $Q$ does not possess a real polynomial factor of positive degree. By
\cite[Theorem 2.3]{huang02}, this can be done without loss of generality since
non-constant univariate real polynomials always admit factorizations with linear
factors over $\H$. For later reference, we state this as a proposition:

\begin{prop}
\label{prop:realfactors}
A non-constant univariate real polynomial $F \in \R[t]$ admits a factorization
with linear factors over $\H$.
\end{prop}

The norm polynomial $\No{Q}$ is a real polynomial and it can be decomposed into irreducible factors over $\R$. If $M \in \R[t]$ is a monic irreducible factor of $\No{Q}$, one can apply division with remainder\footnote{Division with remainder is applicable for polynomials in any polynomial ring $R[t]$, where $R$ is an arbitrary ring, if the leading coefficient of the divisor polynomial is invertible (for instance, c.~f. \cite{li18}). If $R$ is not commutative, we have to distinguish between left-division and right-division with remainder. However, $M \in \R[t]$ is in the center of $\H[t]$, hence this distinction is not necessary in our case and the denotation \emph{division with remainder} is justified.} and obtains
\begin{equation}
  \label{DR}
  Q=TM+S \quad \text{with} \quad T, S \in \H[t], \ \deg(S) < 2.
\end{equation}
If $S=0$, we have $Q=TM$ and $Q$ possesses a real polynomial factor of positive
degree, a case which we already excluded. Therefore, we just have to consider
the case $S \neq 0$: The norm polynomial $\No{S}$ of the remainder $S$ satisfies
the important identity $\No{S}=cM$ with $c \in \R \setminus \{0\}$; indeed, the
polynomial $M$ is an irreducible factor of $\No{Q}$ and
\begin{equation}
\label{normS}
	\No{S}=\No{(Q-TM)}=\No{Q}-Q\Cj{T}M-T\Cj{Q}M+\No{T}M^2.
\end{equation}
Since each term on the right-hand side of Equation~\eqref{normS} is divisible by
$M$, the same is true for $\No{S}$. If $\deg(S)=0$, the polynomial $M$ would not
divide $\No{S}\in \R\setminus\{0\}$ and we conclude $\deg(S) = 1$, $\deg(M) =
2$ and also $\No{S}=cM$ with $c \in \R \setminus \{0\}$.

Writing $S=a(t-h)$ with $a, h \in \H$ and $c=\Cj{a}a$, $M=\Cj{(t-h)}(t-h)$, we
conclude that $(t-h)$ is a right factor of both $S$ and $M$ and hence, by
\eqref{DR}, also a right factor of $Q$. Using polynomial division we find $Q'
\in \H[t]$ such that $Q=Q'(t-h)$. By proceeding inductively with $Q'$ instead of $Q$, we obtain the desired factorization with linear factors.

\begin{rmk}
\label{univunique}
Let us close this section with some important observations:

\begin{enumerate}[(i)]
\item \label{it:deg2} If $Q$ has no real polynomial factor of positive degree,
  the above arguments show that a monic irreducible factor $M$ of $\No{Q}$ is of
  degree two.
\item \label{it:unique} The right factor $t-h$ of $Q$ depends on the irreducible
  polynomial $M \in \R[t]$ of $\No{Q}$. By choosing different irreducible
  factors of the norm polynomial, we obtain different right factors of $Q$. In
  general, there exist $\deg Q!$ different factorizations with linear factors.
\item \label{it:leftfactors} Following above ideas, one successively separates
  linear \emph{right} factors until one obtains the desired factorization of
  $Q$. It is equally possible to split off \emph{left} factors, for example by
  computing right factors of the conjugate polynomial $\Cj{Q}$.
\item \label{it:uniquelydet} Unless $Q$ is divisible by $M$, the right factor
  $t-h$ with $\No{(t-h)}=M$ is uniquely determined \cite[Lemma~3]{hegedus13}.
A similar statement is true for the left factors of~$Q$.
\end{enumerate}	
\end{rmk}

We conclude this section with a word of warning on some algorithmic assumptions made throughout this paper. The presented iterative factorization algorithm for univariate quaternionic polynomials requires in each step as input a real quadratic factor $M$ of the norm polynomial $\No{Q}$. In general, the computation of $M$ is only possible numerically which poses the (yet unsolved) question for a robust numeric version of the factorization algorithm. Similar concerns apply to all factorization algorithms we encounter in this article. This should be kept in mind when reading phrases like ``we can compute all factorizations'' or, ``this algorithm yields a factorization of a type xyz if it exists'' in the remainder of this paper.

\subsection{Bivariate Polynomials}
\label{subsec:bivariate-factorization}

Let us now consider univariate factorizations of bivariate quaternionic
polynomials, that is, factorizations with only univariate linear
factors. Any univariate factorization of a bivariate polynomial $Q \in \H[t,s]$
can be represented by a univariate factorization with monic linear factors
\begin{equation}
\label{univfact}
  Q=a(u_1-h_1)\cdots(u_k-h_k),
\end{equation}
where $u_i \in \{t,s\}$ and $a, h_i \in \H$ for $i=1,\ldots,k$. We usually
consider univariate factorizations of the form \eqref{univfact}.

Univariate factorizations may only occur if the norm polynomial factors as $\No{Q}=PR$ with $P \in \R[t]$ and $R \in \R[s]$. In order to see this, compute
\begin{equation*}
	\No{Q}=a(u_1-h_1)\cdots(u_k-h_k)\Cj{(u_k-h_k)}\cdots \Cj{(u_1-h_1)}\Cj{a}.
\end{equation*}
Factors of the form $(u_i-h_i)\Cj{(u_i-h_i)}$ are real univariate polynomials in $t$
or $s$ and hence part of the center of $\H[t,s]$. Using this fact we obtain
\begin{equation*}
	\No{Q}=\No{a}\No{(u_1-h_1)}\cdots \No{(u_k-h_k)},
\end{equation*}
and $\No{Q}$ turns out to be a product of univariate real polynomials.

\begin{defn}[NFC]
	Let $Q \in \H[t,s]$ be a bivariate quaternionic polynomial. We say that $Q$ satisfies the \emph{necessary factorization condition} (NFC) if $\No{Q}=PR$ with $P \in \R[t]$ and $R \in \R[s]$.
\end{defn}

In order to state results on existence of univariate factorizations, we proceed
similarly to Section~\ref{subsec:univariate-factoriation} and neglect real
polynomial factors of positive degree: We endow the polynomial ring $\H[t,s]$
with the graded lexicographic order. For a polynomial $Q=Q_0+\qi Q_1+\qj Q_2 +
\qk Q_3 \in \H[t,s]$ we define $\mrpf(Q)$ to be the monic greatest common
divisor of the polynomials $Q_0,Q_1,Q_2,Q_3 \in \R[t,s]$.\footnote{The acronym
  $\mrpf$ was introduced in \cite{li18b} and stands for ``monic real polynomial
  factor'' of maximal degree.} We then find a polynomial $Q' \in \H[t,s]$ such
that $Q=\mrpf(Q)Q'$. Obviously, the polynomial $Q'$ does not possess a real
divisor of positive degree. If $Q$ satisfies the NFC, the real polynomial
$\mrpf(Q)$ is a product of univariate real polynomials as well. These
polynomials admit factorizations over $\H$ according to
Proposition~\ref{prop:realfactors}. Therefore, it is sufficient to consider
polynomials $Q \in \H[t,s]$ with $\mrpf(Q)=1$.

The NFC turns out to be sufficient if we consider polynomials $Q \in \Hstar$.
This immediately follows from \cite{Skopenkov19} and we will briefly explain how
to find a factorization:

Consider a non-constant polynomial $Q \in \Hstar$ with $\mrpf(Q)=1$ satisfying
the NFC $\No{Q}=PR$ with $P \in \R[t]$ and $R \in \R[s]$. If $\deg(P)=0$, $Q$ is
a linear polynomial in $\H[s]$, hence there is nothing to show. Otherwise, we
choose a monic irreducible factor $M \in \R[t]$ of the univariate $t$-factor $P$
in the norm polynomial. Let $T, S \in \H[t,s]$ be polynomials with
\begin{equation}
\label{DR2}
   Q=TM+S \quad \text{and} \ S \in \H_{11}.
\end{equation}
Representation \eqref{DR2} can be found by writing $Q=Q_0+sQ_1$ with $Q_0, Q_1 \in \H[t]$ and dividing both, $Q_0$ and $Q_1$, with remainder by $M$.\footnote{If we view $Q$ and $M$ to be univariate polynomials in $t$ having coefficients in $\H[s]$, the polynomial $S \in \H_{11}$ is the unique remainder after division with remainder of $Q$ by $M$. Note that, in this setting, the leading coefficient of $M$ is $1$, thus invertible, and division is possible. The decomposition $Q=Q_0+sQ_1$ was used to highlight the fact $S \in \H_{11}$.} If $S=0$, we have $Q=TM$, whence $Q$ has a real polynomial factor of positive degree. This contradicts our assumption $\mrpf(Q)=1$. Therefore, $S \neq 0$, and we repeat the computations made in Equation~\eqref{normS} and conclude that $M$ divides $\No{S}$. Similar to the univariate case, this is only possible if $\deg(M)=2$ and $\No{S}=HM$ for an appropriate polynomial $H \in \R[s]$. Note that $H$ is indeed univariate since $S \in \H_{11}$. Hence $S$ satisfies the NFC. (We would like to emphasize the resemblance of the univariate and bivariate ideas: In the univariate case we obtained $\No{S}=cM$ with $c \in \R\setminus \{0\}$.)

The remainder polynomial $S$ possesses a very low bi-degree and, therefore, is
easy to handle. Indeed, in the simple case of polynomials in $\H_{11}$
satisfying the NFC, the direct computation of a univariate factorization via the
\emph{Splitting Lemma} by Skopenkov and Krasauskas (c.~f. \cite[Splitting
Lemma~2.6]{Skopenkov19}) is possible. Using their result, one obtains
\begin{equation*}
\begin{aligned}
   \text{\emph{Case 1:}} \quad S &= (t-h)A \quad \text{or}\\
   \text{\emph{Case 2:}} \quad S &= A(t-h),
\end{aligned}
\end{equation*}
with $A \in \H[s]$, $\No{(t-h)}=M$, $\No{A}=H$. Now, one can either separate the
factor $t-h$ on the left or the right of $Q$. Indeed, if \emph{Case 1} is
satisfied, $t-h$ is a left factor of both $S$ and $M$ and, by \eqref{DR2},
also of $Q$. (Note that $TM=MT$ since $M \in \R[t]$ is a real polynomial.)
Otherwise, $t-h$ is a right factor of $S$ and $M$ and hence, again by
\eqref{DR2}, also a right factor of $Q$. Choosing the next irreducible factor of
$P$ and proceeding inductively yields a decomposition with univariate linear
factors.

\begin{rmk}
\label{rmk:starone}
\begin{enumerate}[(i)]
\item \label{it:deg2starone} The fact $\mrpf(Q)=1$ implies that the monic
  irreducible factor $M$ of $P$ is of degree two. It is not unique and a
  different choice for $M$ will, in general, also yield a different left or
  right factor $t-h$.
\item \label{bivunique} Similar to the univariate case, we may conclude that the
  univariate \emph{right} factors with norm polynomial $M$ are, in general,
  unique. Indeed, let $t-h$ be a linear right factor of $Q$ and let $M \coloneqq
  \No{(t-h)}$. If $Q$ is not divisible by $M$, we claim that $h$ is uniquely
  determined. The fact $M \nmid Q$ implies that there exists at least one
  coefficient of $Q$, viewed as an element of $\H[t][s]$, that is not divisible
  by $M$. This coefficient is a univariate polynomial in $\H[t]$ with right
  factor $t-h$. Hence $t-h$ is uniquely determined by Remark \ref{univunique},
  Item~\ref{it:uniquelydet}. The same argument can be applied to \emph{left}
  factors of $Q$. This statement is not only true for $Q \in \Hstar$ but also
  for $Q \in \H[t,s]$. Moreover, the arguments remain true for linear
  $s$-factors of~$Q$.
\end{enumerate}
\end{rmk}

For the sake of completeness and the reader's convenience we describe the
factorization technique for polynomials in $\Hstar$ in Algorithm~\ref{starone}.
The algorithm relies on the factorization technique of \cite{Skopenkov19} and in
particular on their crucial \emph{Splitting Lemma}. Unlike~\cite{Skopenkov19},
our slightly modified version forces the linear factors in $t$ to be
\emph{monic} as described in the current section. It takes as input a
\emph{tuple} of quadratic factors of $P$ rather than $P$ itself in order to
account for the ambiguity mentioned in Remark~\ref{rmk:starone},
Item~\ref{it:deg2starone}. Moreover, we assume $\mrpf(Q)=1$. The adaption of
Algorithm~\ref{starone} to the case $\mrpf(Q) \neq 1$ is straightforward (c.~f.
Proposition~\ref{prop:realfactors}). Since $\H[t] \subseteq \Hstar$, the
algorithm also works for univariate polynomials in $\H[t]$ (c.~f. lines
\ref{alg:univ}-\ref{alg:enduniv}).

Implementation of Algorithm~\ref{starone} and also the later
Algorithm~\ref{alg:multtrick} only requires a few standard ingredients:
Factorization of real polynomials, quaternion algebra, and division with
remainder over polynomial rings. Our examples were computed using a self written
library for the computer algebra system Maple \cite{maple}.

\begin{algorithm}
\caption{Factorization of polynomials in $\Hstar$}
\label{starone}
\begin{algorithmic}[1]
\Input A non-constant polynomial $Q \in \Hstar$ with $\mrpf(Q)=1$, satisfying the NFC $\No{Q}=PR$, where $P \in \R[t]$, $R \in \R[s]$.
\Require A tuple $(M_1,\ldots,M_n)$ of monic irreducible polynomials in $\R[t]$ such that $P=\operatorname{lc}(P)\cdot M_1\cdots M_n$, where $\operatorname{lc}(P)$ is the leading coefficient of $P$.
\Ensure Tuples $F_l = (L_1, \ldots, L_k)$, $F_r = (R_1, \ldots, R_p)$ of
univariate, linear, monic polynomials in $\H[t]$ and a linear polynomial $U \in \H[s]$ such that
$Q = L_1 \cdots L_k \cdot U \cdot R_1 \cdots R_p$.\\
{$i \gets 1$, $F_l \gets (\ )$, $F_r \gets (\ )$ \quad (empty tuples)}\\
{$U \gets Q$}
\While{$i\leq n$}
\State $T \gets$ quotient of division with remainder of $U$ by $M_i$ in $\H[s][t]$.
\State $S \gets$ remainder of division with remainder of $U$ by $M_i$ in $\H[s][t]$.
\State $S_0 \gets \operatorname{coeff}(S,s,0)$,\quad $S_1 \gets \operatorname{coeff}(S,s,1)$
\State $S_{00} \gets \operatorname{coeff}(S_0,t,0)$,\
       $S_{10} \gets \operatorname{coeff}(S_0,t,1)$
\State $S_{01} \gets \operatorname{coeff}(S_1,t,0)$,\
       $S_{11} \gets \operatorname{coeff}(S_1,t,1)$
\If{$S_1=0$} \label{alg:univ}
\State $F_r \gets \operatorname{concat}(t+S_{10}^{-1}S_{00},F_r)$\quad
(concatenation of tuples)
\State $U \gets T \Cj{(t+S_{10}^{-1}S_{00})}+S_{10}$ \label{alg:enduniv}
\Else \quad (Splitting Lemma of \cite{Skopenkov19})
\State $q \gets -S_{10}S_{11}^{-1}$
\State $p \gets S_{00}-S_{10}S_{11}^{-1}S_{01}$
\If{$p=0$}
\State $F_r \gets \operatorname{concat}(t+S_{11}^{-1}S_{01},F_r)$
\State $U \gets T\Cj{(t+S_{11}^{-1}S_{01})}+(s-q)S_{11}$
\Else
\State $F_l \gets \operatorname{concat}(F_l,t+S_{01}S_{11}^{-1})$
\State $U \gets \Cj{(t+S_{01}S_{11}^{-1})}T+S_{11}(s-\Cj{(p^{-1}qp)})$
\EndIf
\EndIf
\State $i \gets i+1$
\EndWhile\\
\Return $F_l$, $U$, $F_r$
\end{algorithmic}
\end{algorithm}

\section{A Multiplication Technique}
\label{sec:multiplication-trick}

As outlined in Section \ref{subsec:bivariate-factorization}, existence of
univariate factorizations of bivariate quaternionic polynomials is exceptional.
They may only occur if the -- rather restrictive -- NFC is satisfied.
However, the Beauregard polynomial \eqref{NFCfails} is irreducible in $\H[t,s]$
even though it satisfies the NFC. Nevertheless, a real multiple of $B$ admits
the desired decomposition with univariate linear factors~\eqref{exmult}.
Motivated by this example, we introduce a \emph{multiplication technique}. We
consider bivariate quaternionic polynomials satisfying the NFC and show that
suitable real polynomial multiples of these polynomials admit a univariate
factorization.

We state our results for \emph{monic} polynomials with respect to the graded
lexicographic order. This can be done without loss of generality: If $Q \in
\H[t,s]$ is not monic, we write $Q=\operatorname{lc}(Q)Q'$, where
$\operatorname{lc}(Q) \in \H$ is the leading coefficient of $Q$, and apply our
results to the monic polynomial $Q' \in \H[t,s]$. If a monic polynomial
satisfies the NFC $\No{Q}=PR$ with $P \in \R[t]$ and $R \in \R[s]$, we may
assume that $P$ and $R$ are monic as well. Moreover, if a monic polynomial
admits a factorization of the form \eqref{univfact}, we may always conclude
$a=1$.

\begin{thm}\label{multtrick}
Let $Q \in \H[t,s]$ be a non-constant monic bivariate quaternionic polynomial
with $\mrpf(Q)=1$ satisfying the NFC $\No{Q} = PR$ for $P \in \R[t], R \in
\R[s]$. There exists $K \in \R[t]$ such that $KQ$ decomposes into monic
univariate linear factors, that is,
\begin{equation*}
	KQ=(u_1-h_1) \cdots (u_k-h_k)
\end{equation*}
with $k \in \mathbb{N}$, $u_i \in \{t,s\}$ and $h_i \in \H$ for $i=1,\ldots,k$.
\end{thm}

\begin{proof}
  For $m,n \in \mathbb{N}_0$ let $(m,n)$ be the bi-degree of $Q$, where $m+n=k$.
  We prove the statement via induction over $n$. For $n \in \{0,1\}$ there is
  nothing to show. Indeed, if $n=0$, the polynomial $Q$ is an element of $\H[t]$
  and we can apply the univariate factorization results stated in Section~\ref{subsec:univariate-factoriation}. If $n=1$, it holds that $Q \in \Hstar$
  and we find a factorization according to Section
  \ref{subsec:bivariate-factorization}. In both cases we obtain the desired
  result by choosing $K=1$.

  Let us now assume $n \ge 2$. We choose a monic irreducible factor $M \in
  \R[s]$ of the univariate $s$-factor $R$ in the norm polynomial. We apply
  division with remainder\footnote{We view divident $Q$ and divisor $M$ as
    univariate polynomials in $s$, having coefficients in $\H[t]$. The leading
    coefficient $1$ of $M$ is invertible so that division is possible.} and
  obtain polynomials $T, S \in \H[t,s]$ such that
\begin{equation}
\label{staroneremainder}
Q=TM+S,
\end{equation}
where $S \in \Hstar$. If $S=0$, we have $Q=TM$, a
contradiction to $\mrpf(Q)=1$. Hence $S \neq 0$, and we repeat the computations
applied in Section~\ref{subsec:univariate-factoriation} and Section~\ref{subsec:bivariate-factorization}:
\begin{equation}
\label{normS2}
	\No{S}=\No{(Q-TM)}=\No{Q}-Q\Cj{T}M-T\Cj{Q}M+\No{T}M^2.
\end{equation}
Since $M$ is a factor of $\No{Q}$, each term on the right-hand side of Equation~\eqref{normS2} is divisible by $M$. This is only possible if $\deg(M)=2$ and $\No{S}=HM$ for an appropriate univariate polynomial $H \in \R[t]$. Hence the remainder polynomial $S$ satisfies the NFC. We apply Algorithm \ref{starone} to $S/\mrpf(S)$ and inductively produce linear $t$-factors on the left-hand side and on the right-hand side of the $s$-factor:
\begin{equation*}
	\frac{S}{\mrpf(S)}=(t-h_1)\cdots(t-h_l)(as+b)(t-k_1)\cdots(t-k_r)
\end{equation*}
with $l,r \in \mathbb{N}_0$ and appropriate quaternions $h_1, \ldots, h_l, a,b,
k_1,\ldots,k_r \in \H$.\footnote{The polynomial $S/\mrpf(S)$ need not possess
  any $t$-factor on the left- or right-hand side of the $s$-factor. In this
  case, $l = 0$ or $r = 0$ and we appeal to the convention that the empty
  product equals~$1$.} For better readability, we use the following
abbreviation: We set $h \coloneqq -ba^{-1}$ and $G \coloneqq a
(t-k_1)\cdots(t-k_r) \cdot \mrpf(S) \in \H[t]$. We then obtain
\begin{equation*}
	S=(t-h_1)\cdots(t-h_l)(s-h)G.
\end{equation*}
Note that $M = \No{(s-h)}$.
Define
\begin{equation*}
	K_1' \coloneqq (t-h_1)\cdots(t-h_l) \in \H[t] \quad \text{and} \quad K_1 \coloneqq \No{K_1'} \in \R[t].
\end{equation*}
Then
\begin{equation*}
\begin{split}
	K_1Q &= \No{K_1'}(TM+S)=K_1'M\Cj{K_1'}T+S\No{K_1'}\\
	&=(t-h_1)\cdots(t-h_l)(s-h)\underbrace{(\Cj{(s-h)}\Cj{K_1'}T+G\No{K_1'})}_{\eqqcolon Q'} \\
	&=(t-h_1)\cdots(t-h_l)(s-h)Q'.
\end{split}
\end{equation*}
The polynomial $Q' \in \H_{\ast (n-1)}$ still satisfies the NFC because the norm polynomial $\No{Q'}$ divides $K_1^2\No{Q}=K_1^2PR$ with $K_1^2P\in \R[t]$, $R \in \R[s]$.

By induction hypothesis, there exists $K_2 \in \R[t]$ such that $K_2Q'$ admits a
univariate factorization. (This is true even if $Q'$ has a non-trivial real
polynomial factor, a case which is formally not covered by the induction
hypothesis. This factor is negligible by Proposition~\ref{prop:realfactors}.)

Define $K \coloneqq K_2K_1 \in \R[t]$. Then
\begin{equation*}
\begin{split}
	KQ=& \ K_2(K_1Q)=K_2((t-h_1)\ldots(t-h_l)(s-h)Q')\\
	=& \ (t-h_1)\ldots(t-h_l)(s-h)K_2Q',
\end{split}
\end{equation*}
which proves the claim.
\end{proof}

\begin{rmk}
	In above proof, we computed a factorization of the remainder polynomial
  $S=(t-h_1)\cdots(t-h_l)(s-h)G$, where $G \in \H[t]$, and then forced the
  factor $(t-h_1)\cdots(t-h_l)(s-h)$ to be a left factor of a suitable real
  multiple of $Q$. Obviously, one may also compute a factorization of the form
  $S=F(s-k)(t-k_1)\cdots(t-k_r)$, where $F \in \H[t]$, and find a real
  polynomial multiple of $Q$ admitting a univariate factorization with right
  factor $(s-k)(t-k_1)\cdots(t-k_r)$.
\end{rmk}

\begin{example}
	Let us precisely investigate the Beauregard polynomial \eqref{NFCfails} of \cite{beauregard93}:
\begin{equation*}
	B = (t^2-\qi)s^2+(2\qj t)s+(\qi t^2-1) \in \H_{22} \text{ with } \No{B}=(t^4+1)(s^4+1)
\end{equation*}
Following the proof of Theorem \ref{multtrick}, we choose an irreducible factor
$M \coloneqq s^2+\sqrt{2}s+1 \in \R[s]$ of $s^4+1$. Division with remainder
yields a remainder $S \in \Hstar$ of the form
\begin{equation*}
	S = (\sqrt{2}\qi + 2\qj t - \sqrt{2}t^2)s+\qi(t^2 + 1) - 1 - t^2,
\end{equation*}
which satisfies the NFC. We compute the following factorization of $S$ via Algorithm~\ref{starone}:
\begin{equation*}
	S = \left(t+\frac{-\qk - \qj}{\sqrt{2}}\right)\left(s + \frac{1 - \qi}{\sqrt{2}}\right)\left(\qj - \qk - \sqrt{2}t\right).
\end{equation*}
Define
\begin{equation*}
	K' \coloneqq \left(t+\frac{-\qk - \qj}{\sqrt{2}}\right) \quad \text{ and } \quad K \coloneqq \No{K'} = (t^2+1).
\end{equation*}
Then
\begin{equation*}
	KB=\left(t+\frac{-\qk - \qj}{\sqrt{2}}\right)\left(s + \frac{1 - \qi}{\sqrt{2}}\right)B'
\end{equation*}
with $B' \in \Hstar$ satisfying the NFC.
Algorithm~\ref{starone} yields one factorization of the form
\begin{equation*}
	B'=\left(t+\frac{1+\qk}{\sqrt{2}}\right)\left(t+\frac{-1+\qk}{\sqrt{2}}\right)\left(s+\frac{-1+\qi}{\sqrt{2}}\right)\left(t+\frac{\qj-\qk}{\sqrt{2}}\right),
\end{equation*}
which shows that factorization \eqref{exmult} can be found by means of the
multiplication technique of Theorem~\ref{multtrick}.
\end{example}

\begin{rmk}
\label{simplify}
Even though our multiplication technique actually yields a factorization of a
real polynomial multiple $KQ$ of $Q$, it is sometimes possible to simplify the
real polynomial $K \in \R[t]$. In the first step of the proof of
Theorem~\ref{multtrick}, we computed polynomials $K_1 \in \R[t]$ and $Q' \in
\H[t,s]$ such that $K_1Q=(t-h_1)\cdots(t-h_l)(s-h)Q'$. It may happen that the
polynomials $K_1$ and $Q'$ share a real polynomial factor of positive degree.
One may therefore replace both $K_1$ and $Q'$ by $K_1/\gcd(K_1,\mrpf(Q'))$ and
$Q'/\gcd(K_1,\mrpf(Q'))$, respectively. By applying this idea in each step of
the algorithm, it might be possible to significantly reduce the degree of $K$.
This is illustrated in the next example.
\end{rmk}

\begin{example}
\label{multtrickex3}
	Consider
  \begin{multline*}
		Q \coloneqq
(\qi(-3t + 1) + t^2 - t - 2)s^2\\+ (\qi(-t^2 + t + 2) + \qj(-2t^2 - 2t) + \qk(2 + 2t) - 3t + 1)s\\
		+ \qi(2t^2 - 2t + 4) + \qj(-t^2 + 2t - 1) + \qk(-t^2 - 3) - 2t - 2.
  \end{multline*}
The polynomial satisfies the NFC
\begin{equation*}
\No{Q} = (t^2 + 1)(t^2 - 2t + 5)(s^2 + 2)(s^2 + 3).
\end{equation*}
We choose an irreducible factor
\begin{equation*}
	s^2+3 \in \R[s],
\end{equation*}
apply division with remainder and find a remainder $S \in \Hstar$ that admits the following factorization:
\begin{equation*}
	S = (t - \qi)\left(t - \frac{2\qi}{5} - \frac{4\qj}{5} + \frac{4\qk}{5} + \frac{3}{5}\right)\left(s - \frac{\qi}{5} - \frac{7\qj}{5} + \qk\right)(-\qi - 2\qj).
\end{equation*}
We define
\begin{equation*}
	K \coloneqq (t^2 + 1)\left(t^2 + \frac{6}{5}t + \frac{9}{5}\right)
\end{equation*}
and obtain
\begin{equation*}
	KQ = (t - \qi)\left(t - \frac{2\qi}{5} - \frac{4\qj}{5} + \frac{4\qk}{5} + \frac{3}{5}\right)\left(s - \frac{\qi}{5} - \frac{7\qj}{5} + \qk\right)Q'
\end{equation*}
for an appropriate polynomial $Q' \in \Hstar$. It turns out that the factor $t^2+1$ is also a factor of $Q'$, whence multiplication with this factor is redundant and we just have to consider the polynomial
\begin{equation*}
	\frac{K}{t^2+1} = t^2 + \frac{6}{5}t + \frac{9}{5}.
\end{equation*}
According to Section \ref{subsec:bivariate-factorization} we find a univariate factorization of $\frac{Q'}{t^2+1}$. Ultimately, we obtain
\begin{multline}
\label{complexfact}
	\left(t^2 + \frac{6}{5}t + \frac{9}{5}\right)Q = (t - \qi)\left(t - \frac{2\qi}{5} - \frac{4\qj}{5} + \frac{4\qk}{5} + \frac{3}{5}\right)\left(s - \frac{\qi}{5} - \frac{7\qj}{5} + \qk\right)\\
	\left(t-\frac{6\qi}{5} + \frac{8\qj}{5} - 1\right)\left(s-\frac{4\qi}{5} - \frac{3\qj}{5} - \qk\right)\left(t - \frac{2\qi}{5} - \frac{4\qj}{5} - \frac{4\qk}{5} + \frac{3}{5}\right).
\end{multline}
\end{example}

The proof of Theorem \ref{multtrick} is constructive. This fact enables us an
algorithmic formulation of the multiplication technique
(Algorithm~\ref{alg:multtrick}). It relies on the factorization technique of
\cite{Skopenkov19} for polynomials in $\Hstar$, which is described in Algorithm~\ref{starone}.

\begin{algorithm}
\caption{Factorization by Multiplication with Real Polynomial}
\label{alg:multtrick}
\begin{algorithmic}[1]
\Input A non-constant monic polynomial $Q \in \H[t,s]$ with $\mrpf(Q)=1$, satisfying the NFC $\No{Q}=PR$, where $P \in \R[t]$, $R \in \R[s]$.
\Require A tuple $(M_1,\ldots,M_n)$ of irreducible polynomials in $\R[s]$ such that $R=M_1\cdots M_n$.
\Ensure A real polynomial $K \in \R[t]$ and a tuple of univariate linear polynomials $L=(L_1,\ldots,L_k)$ such that $KQ=L_1\cdots L_k$.\\
{$i \gets 1$, $K \gets 1$, $L \gets (\ )$ (empty tuple)}\\
{$U \gets Q$}
\While{$i < n$}
\State $T \gets$ quotient of division with remainder of $U$ by $M_i$ in $\H[t][s]$
\State $S \gets$ remainder of division with remainder of $U$ by $M_i$ in $\H[t][s]$
\State \label{line:starone} Compute a factorization $\frac{S}{\mrpf(S)}=(t-h_{i_1})\cdots(t-h_{i_l})(s-h_i)(as+b)(t-k_{i_1})\cdots(t-k_{i_r})$
via Algorithm~\ref{starone}.
\State $h_i \gets ba^{-1}$, $G \gets a(t-k_{i_1})\cdots(t-k_{i_r})\cdot \mrpf(S)$
\State $K_i \gets (t-h_{i_1})\cdots(t-h_{i_l})\Cj{(t-h_{i_l})}\cdots \Cj{(t-h_{i_1})}$
\State $U \gets \Cj{(s-h_i)}\Cj{(t-h_{i_l})}\cdots \Cj{(t-h_{i_1})}T+GK_i$
\State $D \gets \gcd(K_i,\mrpf(U))$
\State $K_i \gets K_i/D$, $U \gets U/D$
\State $L \gets \operatorname{concat}(L,t-h_{i_1},\ldots,t-h_{i_l},s-h_i)$
\State $K \gets K_i\cdot K$
\State $i \gets i+1$
\EndWhile
\State Compute a factorization
$U=(t-h_{n_1})\cdots(t-h_{n_l})(s-h_n)(t-k_{n_1})\cdots(t-k_{n_r})$ via Algorithm~\ref{starone}.
\State $L \gets \operatorname{concat}(L,t-h_{n_1},\ldots,t-h_{n_l},s-h_n,t-k_{n_1},\ldots,t-k_{n_r})$\\
\Return $K, L$
\end{algorithmic}
\end{algorithm}

\begin{rmk}
Let us continue with a few remarks on Algorithm~\ref{alg:multtrick}:
\begin{enumerate}[(i)]
\item Algorithm~\ref{alg:multtrick} depends on a tuple $(M_1,\ldots,M_n)$ of
  irreducible real factors of $R$. This tuple is unique up to permutation. By
  choosing different orders of the irreducible factors, we obtain, similar to
  Remark~\ref{rmk:starone}, Item~\ref{it:deg2starone}, different real multiples
  of $Q$ that admit univariate factorizations. In general, if the polynomials
  $M_1,\ldots, M_n$ are pairwise different, we find $n!$ different univariate
  factorizations of real multiples of $Q$.
\item Obviously, our ideas also work for irreducible $t$-factors of $\No{Q}$. We
  then obtain a polynomial $K \in \R[s]$ such that $KQ$ admits a univariate
  factorization. If $P=N_1\cdots N_m$ is a decomposition of the $t$-factor $P$
  in the norm polynomial with irreducible factors and if $N_1,\ldots, N_m$ are
  different in pairs, the algorithm yields $m!$ different univariate
  factorizations.
\item For a polynomial $Q \in \H_{mn}$ with $\mrpf(Q)=1$ we find up to $m!+n!$
  different univariate factorizations of real multiples of~$Q$.
\end{enumerate}
\end{rmk}

\begin{example}
\label{exmulttrick5}
	Let us again consider the Beauregard polynomial \eqref{NFCfails} from \cite{beauregard93}. It is of bi-degree $(2,2)$. The multiplication technique yields $2!+2!=4$ factorizations of real multiples of the polynomial~$B$:	
\begin{multline*}
	(t^2+1)B=\left(t+\frac{-\qk-\qj}{\sqrt{2}}\right)\left(s+\frac{1-\qi}{\sqrt{2}}\right)\left(t+\frac{1+\qk}{\sqrt{2}}\right)\\ \left(t+\frac{-1+\qk}{\sqrt{2}}\right)\left(s+\frac{-1+\qi}{\sqrt{2}}\right)\left(t+\frac{\qj-\qk}{\sqrt{2}}\right).
\end{multline*}
\begin{multline*}
	(t^2+1)B=\left(t + \frac{\qj+ \qk}{\sqrt{2}}\right) \left(s + \frac{-1+\qi}{\sqrt{2}}\right)\left(t +\frac{1-\qk}{\sqrt{2}}\right) \\ \left(t+\frac{-1-\qk}{\sqrt{2}}\right)\left(s+\frac{1-\qi}{\sqrt{2}}\right)\left(t +\frac{\qk-\qj}{\sqrt{2}}\right)
\end{multline*}
\begin{multline*}
	(s^2+1)B= \left(s + \frac{\qj-\qk}{\sqrt{2}}\right)\left(t +\frac{-1-\qi}{\sqrt{2}}\right)\left(s + \frac{1+\qk}{\sqrt{2}}\right) \\ \left(s+\frac{-1+\qk}{\sqrt{2}}\right)\left(t+\frac{1+\qi}{\sqrt{2}}\right)\left(s +\frac{-\qj-\qk}{\sqrt{2}}\right)
\end{multline*}
\begin{multline*}
	(s^2+1)B=\left(s + \frac{\qk-\qj}{\sqrt{2}}\right)\left(t+\frac{1+\qi}{\sqrt{2}}\right)\left(s +\frac{1-\qk}{\sqrt{2}}\right) \\ \left(s+\frac{-1-\qk}{\sqrt{2}}\right)\left(t +\frac{-1-\qi}{\sqrt{2}}\right)\left(s+\frac{\qj+\qk}{\sqrt{2}}\right)
\end{multline*}
\end{example}

\section{An a Posteriori Condition for Existence of Factorizations}
\label{sec:a-posteriori-condition}

So far we have just considered univariate factorizations of real multiples of
$Q$, where $Q$ satisfies the NFC $\No{Q} = PR$ with $P \in \R[t]$ and $R \in
\R[s]$. We did not yet take the possibility into account that $Q$ itself admits
a univariate factorization. If this is the case, multiplication with a real
polynomial is not necessary. An \emph{a priori} condition for existence of a
univariate factorization of $Q$ is yet unknown. Moreover, if a univariate
factorization exists, we do not know how to find it. At least the second issue
can be tackled by means of our multiplication technique: According to Remark
\ref{simplify}, it may sometimes happen that the real polynomial $K \in \R[t]$
cancels out. We will show that, if a univariate factorization exists, there is a
suitable permutation $(M_1,\ldots,M_n)$ of irreducible $s$-factors of $R$ such
that Algorithm~\ref{alg:multtrick} yields $K = 1$ and a factorization that is
\emph{equivalent} to the given factorization in a sense to be specified.

We proceed by introducing a sensible concept of equivalence of factorizations
that will allow us to formulate our statements in a clear and simple way. It
identifies factorizations that arise from ambiguities of factorizations of
univariate polynomials (c.~f. Remark~\ref{univunique}, Item~\ref{it:uniquelydete}) and
from commutation of adjacent $t$- and $s$-factors:

\begin{example}
  Because of $(t+\qi+\qj)(t-\qi) = (t + \qi)(t - \qi + \qj)$, the polynomial
\begin{equation*}
    Q = (s-\qi)(t+\qi+\qj)(t-\qi)(s-2\qi+\qk)
  \end{equation*}
obviously allows a second factorization:
\begin{equation*}
    Q = (s-\qi)(t + \qi)(t - \qi + \qj)(s-2\qi+\qk).
  \end{equation*}
Moreover, in this special example, we may even write
\begin{equation*}
	Q = (t + \qi)(s-\qi)(t - \qi + \qj)(s-2\qi+\qk)
\end{equation*}
since $(t+\qi)$ and $(s-\qi)$ commute. We will view all of these factorizations
of $Q$ as equivalent.
\end{example}

\begin{defn}
  \label{def:equivalence}
  Given a univariate factorization $Q = (u_1-h_1)\cdots(u_k-h_k)$ of a monic
  polynomial $Q \in \H[t,s]$ with $u_i \in \{t,s\}$ and $h_i \in \H$ for
  $i=1,\ldots,k$, we define two elementary operations that again yield a
  univariate factorization of~$Q$:
  \begin{enumerate}[(i)]
  \item Interchange $u_l-h_l$ and $u_{l+1}-h_{l+1}$, provided these factors
    commute.
  \item Replace the product $(u_l-h_l)(u_{l+1}-h_{l+1})$ by
    $(u_l-k_l)(u_{l+1}-k_{l+1})$ where $u_l = u_{l+1}$ and
    $(u_l-h_l)(u_{l+1}-h_{l+1}) = (u_l-k_l)(u_{l+1}-k_{l+1})$.
  \end{enumerate}
  Two univariate factorizations of a monic polynomial $Q \in \H[t,s]$ are
  called \emph{equivalent}\footnote{Formally, we consider a univariate
    factorization to be an $(m+n)$-tuple of monic linear univariate polynomials, where
    $(m,n)$ is the bi-degree of $Q$. The set of all possible univariate
    factorizations of $Q$ can then be viewed as a subset of $(\H[t,s])^{m+n}$.}
  if they correspond in a sequence of elementary operations.
\end{defn}
The second elementary operation replaces a quadratic $s$- or $t$-factor by its
second factorization according to Remark~\ref{univunique}, Item~\ref{it:unique}.
In \cite[Definition~4]{li18b} this is called a ``Bennett flip'' and the explicit
formula
\begin{equation*}
  k_{l+1} = -(\Cj{h_l}-h_{l+1})^{-1}(h_lh_{l+1}-h_l\Cj{h_l}),\quad
  k_l = h_l + h_{l+1}-k_{l+1}
\end{equation*}
is provided. It is well-known that all factorizations of a univariate
quaternionic polynomial can be generated by Bennett flips
\cite{hegedus13,li18b}.

Definition~\ref{def:equivalence} captures a natural notion for equivalence of
factorizations in our context. Nonetheless, we also introduce a stricter concept
of equivalence which takes into account the asymmetry of
Algorithm~\ref{alg:multtrick} with respect to $s$ and $t$. Let $Q \in \H[t,s]$ be a
monic polynomial with $\mrpf(Q)=1$. Assume that $Q$ admits two factorizations
with univariate linear factors. Write
\begin{equation}
	\label{twofact}
	\begin{split}
		Q &= A_0(s-h_1)A_1(s-h_2)A_3(s-h_3)\cdots A_{n-1}(s-h_n)A_n\\
		& = \widetilde{A}_0(s-\widetilde{h}_1)\widetilde{A}_1(s-\widetilde{h}_2)\widetilde{A}_3(s-\widetilde{h}_3)\cdots \widetilde{A}_{n-1}(s-\widetilde{h}_n)\widetilde{A}_n
	\end{split}
\end{equation}
with monic $A_i, \widetilde{A}_i \in \H[t]$ for $i=0,\ldots,n$ and $h_i,
\widetilde{h}_i \in \H$ for $i=1,\ldots,n$. We use this notation in order to
highlight the appearance of the univariate linear $s$-factors. The polynomials
$A_i$ and $\widetilde{A}_i$ occur by merging consecutive linear $t$-factors.
Some of the polynomials $A_i$ and $\widetilde{A}_i$ may equal $1$.

\begin{defn}
\label{equiv}
We call the two factorizations \eqref{twofact} \emph{$t$-equivalent} if
\begin{equation*}
  \No{(s-h_i)}=\No{(s-\widetilde{h}_i)}
\end{equation*}
for all $i=1,\ldots,n$. This definition
is an equivalence relation on the set of all possible univariate factorizations
of $Q$. Obviously, there is an analogous concept of \emph{$s$-equivalence} of
factorizations.
\end{defn}

\begin{prop}
\label{t-equivalence}
Let $Q \in \H[t,s]$ be a monic polynomial with $\mrpf(Q)=1$. Moreover, assume
that $Q$ admits two factorizations of the form \eqref{twofact} which are
$t$-equivalent according to Definition~\ref{equiv}.
If $\gcd(\No{A_0},\No{\widetilde{A}_0})=1$, we obtain $A_0(s-h_1)=(s-h_1)A_0$,
$\widetilde{A}_0(s-\widetilde{h}_1)=(s-\widetilde{h}_1)\widetilde{A}_0$, and
$h_1 = \widetilde{h}_1$.
\end{prop}

We prove the statement by means of the following technical lemma:

\begin{lem}
\label{helemma}
Let $Q \in \H[t,s]$ be a bivariate quaternionic polynomial of the form $Q = A(s-h)B$
with $A \in \H[t]$, $B \in \H[t,s]$ and $h \in \H$.

The remainder of having divided $Q$ by $M \coloneqq \No{(s-h)}$ equals
$R=A(s-h)C$ for an appropriate $C \in \H[t]$, that is, $Q$ and $R$ share the
common left factor $A(s-h)$.
\end{lem}

\begin{proof}
	Let us divide the polynomial $(s-h)B$ with remainder by $M$. We obtain $(s-h)B=TM+S$ with $S \in \Hstar$. Obviously, $(s-h)$ is a left factor of both $(s-h)B$ and $TM=MT$. For this reason, it is also a left factor of $S$. We write $S=(s-h)C$ for an appropriate polynomial $C \in \H[t]$. Ultimately, we obtain
\begin{equation}
\label{rem}
	Q=A(s-h)B=A(TM+S)=A(TM+(s-h)C)=ATM+A(s-h)C.
\end{equation}
Since $A$ is a univariate polynomial in $\H[t]$, the polynomial $A(s-h)C$ is an
element of $\Hstar$. By Equation~\eqref{rem}, it turns out to be the unique
remainder of having divided $Q$ by $M$, whence the claim follows.
\end{proof}

\begin{proof}[Proof of Proposition~\ref{t-equivalence}]
We apply division with remainder
  of $Q$ by $\No{(s-h_1)}=\No{(s-\widetilde{h}_1)}$. According to
  Lemma~\ref{helemma}, the remainder polynomial $S \in \Hstar$ admits two
  factorizations of the form
\begin{equation}
\label{twofact1}
	S=A_0(s-h_1)B_0=\widetilde{A}_0(s-\widetilde{h}_1)\widetilde{B}_0,
\end{equation}
where $B_0, \widetilde{B}_0 \in \H[t]$. By \cite[Theorem 3.21]{lercher20}, this
implies $A_0(s-h_1) = (s-h_1)A_0$ and $\widetilde{A}_0(s-\widetilde{h}_1) =
(s-\widetilde{h}_1)\widetilde{A}_0$. Plugging this into \eqref{twofact1} we
obtain $(s-h_1)A_0B_0 = (s-\widetilde{h}_1)\widetilde{A}_0\widetilde{B}_0$. Now
$h_1 = \widetilde{h}_1$ by Remark~\ref{rmk:starone}, Item~\ref{bivunique}. This
concludes the proof.
\end{proof}

\begin{prop}
  \label{equivalence}
  Two factorizations that are $t$-equivalent are also equivalent in the sense of
  Definition~\ref{def:equivalence}.
\end{prop}

\begin{proof}
  We argue that $t$-equivalence of the two factorizations \eqref{twofact} indeed
  implies that they can be converted into each other by computing suitable
  factorizations of the polynomials $A_i$, $\widetilde{A}_i$ and by commuting
  neighbouring $s$- and $t$-factors: As long as $\No{A_0}$ and
  $\No{\widetilde{A}_0}$ share common irreducible factors
we choose a quadratic irreducible common factor $M$ of $\No{A_0}$ and
  $\No{\widetilde{A}_0}$ and compute left factors of both $A_0$ and
  $\widetilde{A}_0$ with norm $M$ according to Remark \ref{univunique},
  Item~\ref{it:leftfactors}. These factors are also left factors of $Q$ and the
  fact $\mrpf(Q)=1$ implies $M \nmid Q$. According to Remark \ref{rmk:starone},
  Item~\ref{bivunique}, the respective linear factors are equal.

  According to Proposition~\ref{t-equivalence}, the remaining factors $B_0$,
  $\widetilde{B}_0$ of $A_0$ and $\widetilde{A}_0$, respectively, commute with
  $(s-h_1)=(s-\widetilde{h}_1)$. Now, we apply the same arguments to the new
  polynomials $B_1 \coloneqq B_0A_1$, $\widetilde{B}_1 \coloneqq
  \widetilde{B}_0\widetilde{A}_1$ and $(s-h_2), (s-\widetilde{h}_2)$.
  Inductively, the claim follows.
\end{proof}

\begin{example}
	\label{tworepresentatives}
  The two factorizations
\begin{equation*}
		\left(t - \frac{7\qi}{5} + \frac{\qk}{5}\right)\left(t - \frac{3\qi}{5} + \frac{4\qk}{5}\right)(s + 2\qi - 2\qk)(t + 2\qj)(s - \qi - 4\qj + \qk)(t + \qj - 2\qk)
	\end{equation*}
and
	\begin{multline*}
		(t - \qi)(s + 2\qi - 2\qk)\left(t - \frac{4\qi}{3} + \frac{2\qj}{3} + \frac{4\qk}{3}\right)(s - \qi - 4\qj + \qk)\\
		\left(t + \frac{14\qi}{33} + \frac{65\qj}{33} - \frac{32\qk}{33}\right)\left(t - \frac{\qi}{11} + \frac{4\qj}{11} - \frac{15\qk}{11}\right)
	\end{multline*}
  are $t$-equivalent and, by Proposition~\ref{equivalence}, also equivalent. The
  reader is invited to transform them into each other by using only elementary
  operations according to Definition~\ref{def:equivalence}.
\end{example}

\begin{example}
\label{ex:eq-not-teq}
  The polynomial
  \begin{equation}
    \label{sixfact}
    Q = (t - \qi)(t - \qj)(s - \qj)(s - \qi + \qj)(s - 2\qk)(t - 1 - \qj)
  \end{equation}
  admits six univariate factorizations arising from the six different
  factorizations of $(s-\qj)(s - \qi + \qj)(s - 2\qk)$ according to
  Section~\ref{subsec:univariate-factoriation}. None of them are $t$-equivalent.
  Among these factorizations, two have the left factor $(s-\qj)$ which commutes
  with $(t-\qj)$. Each of them gives rise to a further $t$-equivalent
  factorization. All eight factorizations are equivalent, showing that the converse
  of Proposition~\ref{equivalence} is not true. From Corollary~\ref{maxnumfact}
  below it will follow that the polynomial admits no further factorizations.
\end{example}

The following theorem is the centerpiece of the present section:

\begin{thm}
\label{exfact}
Let $Q \in \H[t,s]$ be a monic quaternionic polynomial 
satisfying the \emph{NFC} $\No{Q}=PR$, where $P \in \R[t]$, $R \in \R[s]$, and
$\mrpf(Q) = 1$.
Moreover, let $R=M_1\cdots M_n$ be a decomposition of $R$ with monic, quadratic,
irreducible factors. If the polynomial $Q$ admits a univariate factorization,
there exists a permutation $\sigma \in S_n$ such that Algorithm~\ref{alg:multtrick}
with input $(M_{\sigma(1)},\ldots,M_{\sigma(n)})$ yields $K=1$ and a univariate
factorization which is equivalent to the given factorization of~$Q$.
\end{thm}

\begin{proof}
By assumption, the polynomial $Q$ decomposes into univariate linear factors. We are interested in the leftmost $s$-factor of this factorization. For this reason, we write
\begin{equation}
\label{findfact}
	Q = (t-h_1)\cdots(t-h_l)(s-h)\widetilde{Q}
\end{equation}
with $l \in \mathbb{N}_0$, $h_1, \ldots, h_l, h \in \H$ and an appropriate polynomial $\widetilde{Q} \in \H[t,s]$. (In case $l=0$, the empty product convention applies.) There exists $k \in \{1,\ldots,n\}$ such that
\begin{equation*}
	\No{(s-h)}=M_k.
\end{equation*}
We set $\sigma(1) \coloneqq k$ and apply the first step of
Algorithm~\ref{alg:multtrick}. By Lemma \ref{helemma}, the remainder polynomial
of having divided $Q$ by $M_{\sigma(1)}$ equals
\begin{equation*}
	S=(t-h_1)\cdots(t-h_l)(s-h)C
\end{equation*}
for an appropriate $C \in \H[t]$. Following the proof of Theorem \ref{multtrick}, we have to define	
\begin{equation*}
  K_1 \coloneqq (t-h_1)\cdots(t-h_l)\Cj{(t-h_l)} \cdots \Cj{(t-h_1)}
\end{equation*}
and obtain
\begin{equation*}
	K_1Q = (t-h_1)\cdots(t-h_l)(s-h)Q'
\end{equation*}
with $Q' \in \H[t,s]$. By Equation~\eqref{findfact},
$(t-h_1)\cdots(t-h_l)(s-h)$ is a left factor of $Q$. Therefore, $K_1$ is a real
divisor of $Q'$ (we have $Q'=K_1\widetilde{Q}$) and we get rid of
this divisor according to Remark \ref{simplify}. Replacing $Q$ by $\widetilde{Q}$ and
proceeding inductively yields the desired factorization of~$Q$.
\end{proof}

\begin{rmk}
  In contrast to Algorithm~\ref{alg:multtrick}, our proof of Theorem~\ref{exfact}
  yields precisely the given factorization of $Q$. Algorithm~\ref{alg:multtrick} is
  not deterministic (for example, line~\ref{line:starone} leaves us a choice).
  Hence, its output is not necessarily identical but $t$-equivalent and, by
  Proposition~\ref{equivalence}, also equivalent to the given factorization
  of~$Q$.
\end{rmk}

Theorem \ref{exfact} provides an \emph{a posteriori} condition for existence of univariate factorizations. In case of existence of a univariate factorization, at least an equivalent univariate factorization can be found by application of the multiplication technique.

\begin{example}
\label{onefact}
	Consider the polynomial $Q$ from Example \ref{multtrickex3}. In Equation~\eqref{complexfact}, we found a factorization of $(t^2+\frac{6}{5}t+\frac{9}{5})Q$ with univariate linear factors by applying Algorithm~\ref{alg:multtrick} with input $(s^2+3,s^2+2)$. In fact, it turns out that already $Q$ admits a univariate factorization. It can be found by application of Algorithm~\ref{alg:multtrick} with input $(s^2+2,s^2+3)$:
\begin{equation*}
	Q = (t - \qi)(s - \qj + \qk)(t - 2\qi - 1)(s - \qi - \qj - \qk).
\end{equation*}
\end{example}

\begin{example}
	The Beauregard polynomial $B$ does not admit a
  factorization with univariate linear factors since it is irreducible in
  $\H[t,s]$. Nonetheless, non-existence of a univariate factorization also
  follows from the fact that Algorithm~\ref{alg:multtrick} with inputs
  $(s^2+\sqrt{2}s+1,s^2-\sqrt{2}+1)$ or $(s^2-\sqrt{2}s+1,s^2+\sqrt{2}s+1)$ only
  yields univariate factorizations of $(t^2+1)B$ (c.~f. Example \ref{exmulttrick5}).
\end{example}

\begin{cor}
  \label{maxnumfact}
	Let $Q \in \H[t,s]$ satisfy the assumptions of Theorem \ref{exfact}. Moreover,
  let $(m,n)$ with $m, n \in \mathbb{N}_0$ be the bi-degree of $Q$. The
  polynomial $Q$ may admit up to $k!$ non-equivalent univariate factorizations, where $k \coloneqq \min(m,n)$.
  All of them can be found by applying the multiplication technique.
\end{cor}

\begin{proof}
Let us consider the case $n \leq m$. By Theorem \ref{exfact}, any univariate factorization of $Q$ is equivalent to a factorization that can be found by application of Algorithm~\ref{alg:multtrick} with suitable input polynomials. However, there exist at most $n!$ different tuples with quadratic,
  irreducible, real polynomials $(M_1,\ldots,M_n)$, which can be used as input for Algorithm~\ref{alg:multtrick}. This proves the claim. If $m \leq n$, we interchange $s$ and $t$.
\end{proof}

\begin{rmk}
An arbitrary input tuple need not automatically yield a univariate factorization
of $Q$. This is only the case if the algorithm produces $K=1$. Moreover, it may
happen that different input tuples of Algorithm~\ref{alg:multtrick} give rise to
univariate factorizations which are not $t$-equivalent, but still equivalent.
Indeed, consider the polynomial $Q$ from Example~\ref{ex:eq-not-teq}.
Algorithm~\ref{alg:multtrick} with different input tuples yields six different
univariate factorizations. All of them are equivalent, demonstrating that the upper bound of Corollary~\ref{maxnumfact} need not be strict.
\end{rmk}

\begin{example}
Consider the polynomial $Q$ from Examples~\ref{multtrickex3} and \ref{onefact}. By Corollary \ref{maxnumfact}, all univariate factorizations are equivalent to the factorization obtained by application of Algorithm~\ref{alg:multtrick} with input $(s^2+2,s^2+3)$. This is the only input tuple which yields a univariate factorization.
\end{example}

\section{A Remarkable Example and Future Research}
\label{sec:remarkable-example}

We conclude this article by a remarkable example which shows that the upper
bound of Corollary~\ref{maxnumfact} is sharp. It also demonstrates applicability
of our factorization ideas to mechanism science which is a topic of future research.

\begin{example}
	The following polynomial $Q \in \H[t,s]$ admits $2!=2$ non-equivalent
  factorizations with univariate linear factors:
\begin{equation}
    \label{eq:remarkable}
	\begin{split}
		Q & = (t + \qi + \qj + 2\qk)(s + \qk)(t -\qi - \qj)(s+\qi + \qj - \qk)\\
		& = (s + \qi + \qj + \qk)(t+\qi + \qj)(s - \qk)(t - \qi - \qj + 2\qk).
	\end{split}
	\end{equation}
Both of them can be found by application of the multiplication technique. To
  the best of our knowledge, this is the first example of a bivariate
  quaternionic polynomial with non-equivalent factorizations and without a real
  polynomial factor of positive degree.
\end{example}

One of the prime applications of quaternions is kinematics. Quaternions can be used to
model the special orthogonal group $\SO$ and linear quaternionic polynomials parametrize
rotations around fixed axes \cite{hegedus13}. Thus, the two factorizations in
\eqref{eq:remarkable} describe, in two ways, a spherical two-parametric motion
of a chain of revolute joints. Note that non-equivalence of factorizations is
crucial as otherwise the two chains of revolute joints ``essentially'' coincide
by arguments similar to \cite{lercher20}.

Quite surprisingly, the two factorizations \eqref{eq:remarkable} (and all other
similar examples that we know of) can be extended to dual quaternions which
allows to extend the spherical chains of revolute joints to spatial mechanisms (c.~f. \cite{hegedus13}).

The dual quaternions $\DH$ are obtained by adjoining a new element $\varepsilon$
to $\H$ which commutes with everything and squares to zero: $\varepsilon^2 = 0$.
The algebra of dual quaternions is isomorphic to the even subalgebra
$\mathcal{C}\ell^{+}_{(3,0,1)}$ of the Clifford algebra
$\mathcal{C}\ell_{(3,0,1)}$, which was introduced in Section
\ref{sec:preliminaries}. The generators of $\mathcal{C}\ell^{+}_{(3,0,1)}$ are
the elements $e_0$, $e_{12}$, $e_{13}$, $e_{14}$, $e_{23}$, $e_{24}$, $e_{34}$ and $e_{1234}$. The isomorphism is given by (c.~f. \cite[p.~184]{klawitter15})
\begin{equation*}
  \begin{split}
     e_0 &\mapsto 1, \quad  e_{23} \mapsto \qi, \quad e_{31} \mapsto \qj, \quad e_{12} \mapsto \qk,\\
        -e_{1234} &\mapsto \varepsilon, \quad e_{14} \mapsto \varepsilon \qi, \quad e_{24} \mapsto \varepsilon \qj, \quad e_{34} \mapsto \varepsilon \qk.
  \end{split}
\end{equation*}
An element $h \in \DH$ can be
uniquely written as $h = p + \varepsilon d$ with primal part $p \in \H$ and dual part
$d \in \H$.

In order to extend the factorization \eqref{eq:remarkable} to $\DH$ we make the
ansatz
\begin{equation}
  \label{eq:remarkable-dual}
  C = G_1G_2G_3G_4 = H_1H_2H_3H_4
\end{equation}
where
\begin{alignat*}{2}
  G_1&= t + \qi + \qj + 2\qk + \varepsilon d_1,\qquad&
  H_1 &= s + \qi + \qj + \qk + \varepsilon f_1,\\
  G_2&= s + \qk + \varepsilon d_2,\qquad&
  H_2 &= t+\qi + \qj + \varepsilon f_2,\\
  G_3&= t -\qi - \qj + \varepsilon d_3,\qquad&
  H_3 &= s - \qk + \varepsilon f_3,\\
  G_4&= s+\qi + \qj - \qk + \varepsilon d_4,\qquad&
  H_4 &= t - \qi - \qj + 2\qk + \varepsilon f_4,
\end{alignat*}
with yet unknown quaternions $d_\ell$, $f_\ell \in \H$. A linear polynomial $t +
p + \varepsilon d \in \DH[t]$ parametrizes a rotation around an arbitrary axis
in space if it satisfies the Study condition $p\Cj{d} + d\Cj{p} = 0$ and $d +
\Cj{d} = 0$ \cite{husty12}. Imposing these conditions on each linear factor and
augmenting this system of equations with the conditions obtained by comparing
coefficients of $s$ and $t$ on both sides of \eqref{eq:remarkable-dual} yields a
system of 48 linear equations for 32 unknowns. The solution space turns out to
be of dimension four:
\begin{gather*}
  \begin{aligned}
    d_1 &= (\beta-\alpha+\delta-2\gamma)\qi+(\alpha - \beta-\delta)\qj+\gamma\qk, \\
    d_2 &= (\delta-\gamma)\qi+\alpha\qj, \\
    d_3 &= -\beta\qi+\beta\qj-(\alpha+\gamma+\delta)\qk, \\
    d_4 &= \frac{1}{2}(\alpha+2\beta-\delta+2\gamma)\qi-\frac{1}{2}(\alpha+2\beta-\delta)\qj+\gamma\qk, \\
    f_1 &= \frac{1}{2}(-\alpha+2\beta+\delta-2\gamma)\qi+\frac{1}{2}(\alpha-2\beta-\delta)\qj+\gamma\qk, \\
    f_2 &= \beta\qi-\beta\qj-(\alpha-\gamma+\delta)\qk, \\
    f_3 &= (\alpha+\gamma)\qi+\delta\qj, \\
    f_4 &= (-\beta-\alpha+\delta-2\gamma)\qi+(\alpha+\beta-\delta)\qj-\gamma\qk,
  \end{aligned}\\
  \alpha, \beta, \gamma, \delta \in \R.
\end{gather*}

Each of the two factorizations in \eqref{eq:remarkable-dual} yields an open
chain of revolute joints which can follow the motion parametrized by $C \in
\DH[t,s]$. However, this requires synchronization of joints that share the same
parameter values $t$ or $s$. In order to avoid this control problem, one can
combine the two open chains to obtain a closed-loop spatial mechanism with eight
revolute joints (Figure~\ref{fig:1}; axes are labeled by the factors in
\eqref{eq:remarkable-dual}) and remarkable properties. As any generic mechanism
of this type, it has two degrees of freedom. Its configuration variety contains
the motion parametrized by the polynomial $C \in \DH[t,s]$ in a very special
way. By construction, locking any of its joints (parametrized by $t$, say)
automatically locks every other joint parametrized by $t$. The remaining four
joints yield, in any configuration, a closed-loop spatial structure with four
revolute joints which one would expect to be rigid but which is movable (via
parameter $s$) at any configuration (for any $t$). A closer investigation of
this mechanism is on the agenda for future research.

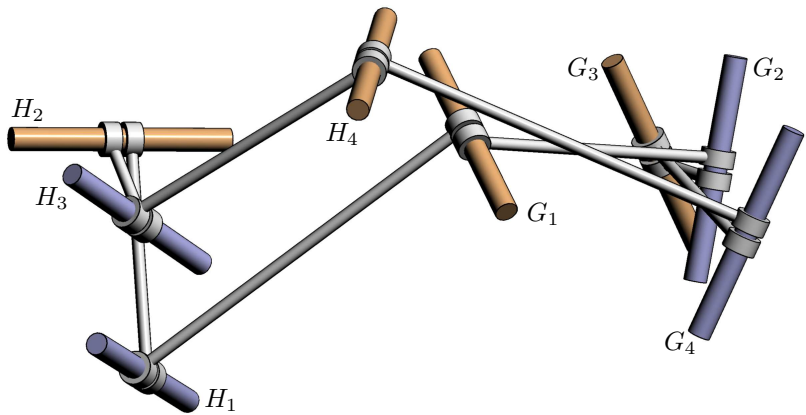
\begin{figure}
  \centering
  \begin{overpic}{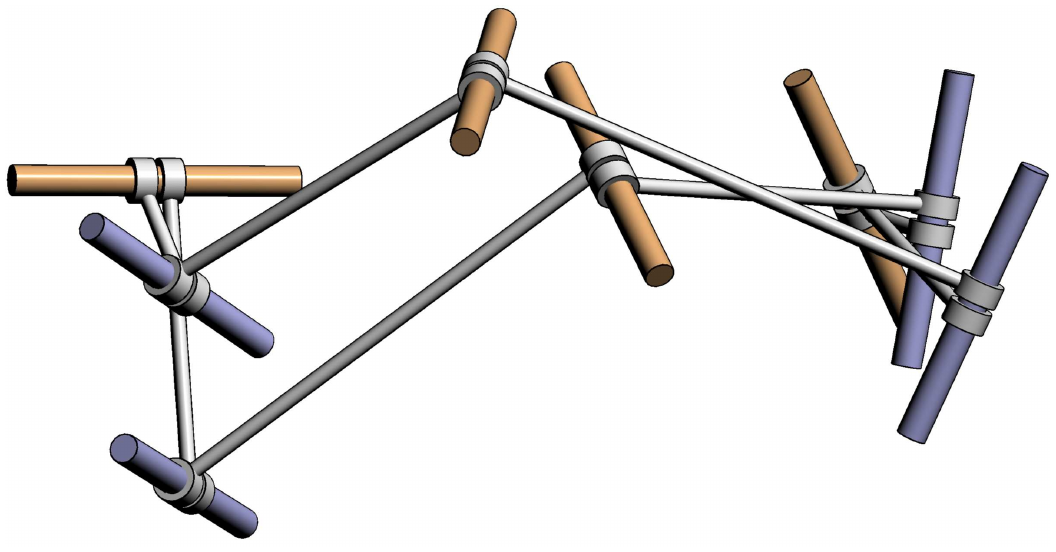}
    \put(65,25){$G_1$}
    \put(93,43){$G_2$}
    \put(70,43){$G_3$}
    \put(82,9){$G_4$}
    \put(25,2){$H_1$}
    \put(1,38){$H_2$}
    \put(4,27){$H_3$}
    \put(40,35){$H_4$}
  \end{overpic}
  \caption{A spatial mechanism constructed from the factorizations in Equation~\eqref{eq:remarkable-dual}}
  \label{fig:1}
\end{figure}

Note that a similar mechanism construction from the polynomial $Q$ in
\eqref{eq:remarkable} is possible but results in a fairly trivial spherical
mechanism with five degrees of freedom whose configuration space contains an
open subset of $\SO$. This demonstrates the need for extending our results to
dual quaternions. Since the algebra of dual quaternions contains zero divisors
and non-invertible elements, factorization theory of bivariate dual quaternionic
polynomials will be more involved.

\end{document}